\documentclass{amsart}
\usepackage[a4paper, 
          bindingoffset=0.2in,
            left=1in,
            right=1in,
            top=1in,
            bottom=1in,
            footskip=.25in]{geometry}
\usepackage[T1]{fontenc}
\usepackage[utf8]{inputenc}
\usepackage{bussproofs}
\usepackage{graphicx}
\usepackage{amsthm}
\usepackage{amsmath}
\usepackage{amsfonts}
\usepackage{amssymb}
\usepackage{amsbsy}
\usepackage{fancyhdr}
\usepackage{mathrsfs}
\usepackage{lmodern}
\usepackage{graphicx}
\usepackage{xcolor, enumerate}
\usepackage[colorlinks,allcolors=blue]{hyperref}
\usepackage[nameinlink,capitalize]{cleveref}
\usepackage{tikz-cd}
\usepackage{mathtools}
\usepackage{parskip}
\usepackage{enumitem}
\usetikzlibrary{positioning}
\usepackage{tikz}
\usetikzlibrary{patterns}
\usepackage[absolute,overlay]{textpos}

\newcommand{\F}{\mathbb{F}}
\newcommand{\PP}{\mathbb{P}}

\newcommand{\C}{\mathcal{C}}

\newcommand{\LL}{\mathcal{L}}

\newcommand{\CL}{\mathcal{C}_\mathcal{L}}

\newcommand{\cT}{\mathcal{T}}

\def\fq{\mathbb{F}_{q}}
\def\fqs{\mathbb{F}_{q^2}}

\def\cF{\mathcal{F}}

\def\cX{\mathcal{X}}
\def\cC{\mathcal{C}}
\def\cF{\mathcal{F}}

\def\cX{\mathcal{X}}

\def\Z{\mathbb{Z}}
\def\cD{\mathcal{D}}
\def\cF{\mathcal{F}}

\def\ff{\mathbb{F}}

\DeclareMathOperator{\lmd}{lmd}
\DeclareMathOperator{\Div}{Div}
\DeclareMathOperator{\Diff}{Diff}

\DeclareMathOperator{\Supp}{Supp}

\DeclareMathOperator{\res}{res}

\DeclarePairedDelimiter\ceil{\lceil}{\rceil}
\DeclarePairedDelimiter\floor{\lfloor}{\rfloor}

\makeatletter
\renewcommand*\env@matrix[1][*\c@MaxMatrixCols c]{%
  \hskip -\arraycolsep
  \let\@ifnextchar\new@ifnextchar
  \array{#1}}
\makeatother

\newtheorem{definition}[subsection]{Definition}

\newtheorem{lemma}[subsection]{Lemma}
\newtheorem{theorem}[subsection]{Theorem}

\newtheorem{proposition}[subsection]{Proposition}
\newtheorem{corollary}[subsection]{Corollary}

\begin{document} 
	\numberwithin{equation}{section}
	\title{Sequences of LCD AG codes and LCP of AG Codes attaining the Tsfasman-Vladut-Zink bound}
	\author[A.\@ Marques,  L.\@ Quoos]{Adler Marques, Luciane Quoos}
	\dedicatory{May 2025}
	\thanks{The first and second authors were partially financed by Coordenação de Aperfeiçoamento de Pessoal de Nível Superior - Brasil (CAPES) - Finance Code 001.
	The second author was also financed by Conselho Nacional de Desenvolvimento Cientíifico e Tecnológico - Brasil (CNPq) - Project 307261/2023-9, and Fundação Carlos Chagas Filho de Amparo à Pesquisa do Estado do Rio de Janeiro - Brasil (Faperj) - Project E-26/204.037/2024.}
	
	\address{Adler Marques, Universidade Federal do Rio de Janeiro, Centro de Tecnologia, Cidade Universitária, Av.\@ Athos da Silveira Ramos 149, Ilha do Fundão, CEP~21.941-909, Brazil. {\it E-mail: adler@im.ufrj.br}}

	\address{Luciane Quoos, Universidade Federal do Rio de Janeiro, Centro de Tecnologia, Cidade Universitária, Av.\@ Athos da Silveira Ramos 149, Ilha do Fundão, CEP~21.941-909, Brazil. {\it E-mail: luciane@im.ufrj.br}}

\begin{abstract}
	Since Massey introduced linear complementary dual (LCD) codes in 1992 and Bhasin et al. later formalized linear complementary pairs (LCPs) of codes -- structures with important cryptographic applications -- these code families have attracted significant interest. We construct infinite sequences  $(C_i)_{i \ge 1}$ of LCD codes and of LCPs $(\C', \cD')_{i \ge 1}$ over $\F_{q^2}$ obtained from the Garcia-Stichtenoth tower of function fields, where we describe suitable non-special divisors of small degree on each level of the tower. These families attain the Tsfasman–Vlăduţ–Zink bound and, for sufficiently large $q$ exceed the classic Gilbert–Varshamov bound, providing explicit asymptotically good constructions beyond existential results. We also exhibit infinite sequences of self-orthogonal over $\F_{q^2}$ and, when $q$ is even, self-dual codes from the same tower all meeting the Tsfasman–Vlăduţ–Zink bound.
\end{abstract}

\maketitle

\section{Introduction}

Let $\fq$ be a finite field with $q$ elements. A linear code $\C$ over  $\fq$ is an  $\fq$-vector subspace of $\fq^n$. Associated to a linear code we have three  fundamental parameters, the length $n=n(\C)$, dimension $k=k(\C)$, and minimum (Hamming) distance $d=d(\C)$.
We say the code $\C$ is an  $[n, k ,d]$-code.

Let $\cC$ be a linear $[n, k, d]$-code over $\fq$. The ratios 
	$$R(\cC ) = \tfrac{k(\cC )}{n(\cC)}=\tfrac{k}{n} \qquad \text{ and } \qquad \delta (\cC ) = \tfrac{d(\cC)}{n(\cC)}=\tfrac{d}{n},$$ 
 denote the \textit{information rate}
and the \textit{relative minimum distance} of the code, respectively. Let $U_q \subseteq [0,1] \times [0,1]$ be defined as the set of points $(R, \delta) \in \mathbb{R}^2$ such that there exists a sequence of codes $(\C_i)_{i \ge 1}$ satisfying 
\[
n(\C_i) \to \infty, \quad R(\C_i) \to R \quad \mbox{and} \quad \delta(\C_i) \to \delta \quad \mbox{as} \quad  i \to \infty.
\]

In \cite{Manin81}, Y. Manin showed that there is a non-increasing and continuous function $\alpha_q \colon [0,1] \to [0,1]$ defined by
\[ 
\alpha_q(\delta) = \sup\{ R \mid (\delta, R) \in U_q, \mbox{ for } \delta \in [0,1] \, \}
\]
such that $\alpha_q(0)=1$, and $\alpha_q(\delta)=0$, for $1-q^{-1} \le \delta \le 1$.
A lower bound for $\alpha_q(\delta)$ was given by Gilbert-Varshamov, \cite[Prop.VII.2.3]{STICH2008}, and the exact value of $\alpha_q(\delta)$ is still a topic of investigation. Let $H_q(\delta)$ be the \textit{$q$-ary entropy function} $H_q \colon [0, 1-q^{-1}] \to \mathbb{R}$ given by  $H_q(0)=0$, and for  $0 < x \le 1 - q^{-1}$
\[
H_q(x) := x \log_q(q-1) - x \log_q(x) - (1-x) \log_q(1-x)
\]

The \textit{Gilbert-Varshamov bound} lower bound for $\alpha_q(\delta)$ states that 
\begin{equation}
\label{eq:GV_bound}
\alpha_q(\delta) \ge 1 - H_q(\delta) \quad \mbox{ for all } 0 < \delta < 1-q^{-1}.
\end{equation}

In 1982, using algebraic geometry codes from algebraic curves
over finite fields, Tsfasman-Vladut-Zink established a new lower bound for $\alpha_q(\delta)$, see \cite{TVZ82}.
Let $\cX$ be an algebraic absolutely irreducible non-singular curve of genus $g$ over $\fq$. Denote by $\cX(\F_q)$ the number of $\fq$ rational points on $\cX$ and
$$N_q(g):=\max_{\cX} \# \cX(\F_q),$$
where $\cX$ runs over the curves defined over $\fq$ of (fixed) genus $g$. The asymptotic behavior  of such curves over the fixed field $\fq$ when the genus goes to infinity is denoted by
 $$A(q):=\limsup_{g \to \infty} N_q(g)/g.$$ 
 The   \textit{Drinfeld-Vlăduţ bound}  states that $A(q) \le \sqrt{q}-1$, see \cite{DV1983}.
If $\fq$ is a finite field such that $A(q) >1,$ then for each real number satisfying $0 \leq  \delta \leq 1$,   Tsfasman-Vlăduţ-Zink  showed that
 \begin{equation} 
    \label{eq:TVZ_bound1}
    \alpha_q(\delta) \geq 1 - \delta -A(q)^{-1},
    \end{equation}
 The same authors also proved that  for $q = \ell^2$ the equality $A(\ell^2) =  \ell -1$ occurs. The lower bound above on the function 
$\alpha_q(\delta)$  was an astonishing result for coding theorists,
since for $q >49$ a square, it gives  an improvement on the Gilbert-Varshamov bound for values of $\delta $  in a certain small interval. As a consequence, for $q=\ell^2$ a square and for all $0 \le \delta \le 1 - (\ell - 1)^{-1}$, there exists an asymptotically good sequence of algebraic geometry codes $(\C_i)_{i \ge 1}$ over $\mathbb{F}_{\ell^2}$ such that 
    $\lim_{i \to \infty} R(\C_i)=R$ and $\lim_{i \to \infty} \delta(\C_i)=\delta$
    satisfying
    \begin{equation} \label{eq:TVZ_bound}
    R + \delta \ge 1 - \frac{1}{\ell - 1}, 
    \end{equation}
 the so-called Tsfasman-Vlăduţ-Zink bound (TVZ-bound).
This leads to the following natural definition. An infinite sequence $\{\mathcal{C}_i\}_{i \ge 1}$ of $[n_i,k_i,d_i]$-codes over $\ff_q$ is said to {\it attain} the
TVZ-bound if, for  $0< \delta\leq 1-(\ell-1)^{-1}$, we have
	$$ \liminf_{i \rightarrow \infty} R(\mathcal{C}_i) \geq 1- \frac{1}{\ell-1}-\delta \qquad \text{and} \qquad \liminf_{i \rightarrow \infty} \delta(\mathcal{C}_i) \geq  \delta. $$ 
	
  An infinite sequence $\{\mathcal{C}_i\}_{i\ge 1}$ of $[n_i,k_i,d_i]$-codes over $\ff_q$ is called \textit{(asymptotically) good} if the information rate $R(\mathcal{C}_i)$ and the relative minimum distance $\delta(\mathcal{C}_i) $ of $\mathcal{C}_i$ satisfy 
 	\[ \liminf_{i \rightarrow \infty} R(\mathcal{C}_i) > 0 \qquad \text{and} \qquad \liminf_{i \rightarrow \infty} \delta(\mathcal{C}_i)>0. \]

 Surprisingly, using the theory of towers of function fields, Garcia and Stichtenoth provided the first explicit construction of a sequence of codes attaining the Tsfasman-Vlăduţ-Zink bound, see \cite{GS1995}.
Since then, the study of structured codes meeting the TVZ-bound has become a fruitful field of research.

At first we recall the definition of certain well studied families of codes. 
If $\C \subseteq \F_{q}^n$  is a code we define the dual code of $\C$ by
$$\C^\perp := \{\textbf{y} \in \F_{q}^n \mid  \langle \textbf{y}, \textbf{x} \rangle=0 \text{ for all } \textbf{x} \in  \C\},$$
where $\langle \cdot, \cdot \rangle$ denotes the Euclidean inner product. A code $\C$ is  self-orthogonal if $\C\subseteq \C^\perp$, and self dual if $\C= \C^\perp$.
A pair $(\C, \cD)$ of linear codes over $\fq$ of length $n$ is called a {\it linear complementary pairs of codes}  (LCP)  if 
$$\C + \cD = \fq^n\quad \text{and} \quad \C \cap \cD = \{ 0 \}.$$ 
In the special case when $\cD = \C^\perp$, $\C$ is a  {\it linear complementary dual  code} (LCD). LCD codes were introduced by J. Massey in 1992, providing a coding solution for the two-user binary additive channel \cite{massey1992}.
In 2015, Bhasin et al. \cite{bhasinetal2015} introduced the term linear complementary pair of codes (LCP of codes). 
    LCPs of codes have been widely studied because of their applications in cryptography, mainly in the protection of symmetric cryptosystem implementations from side channel attacks (SCA) and fault injection attacks (FIA) \cite{bhasinetal2015, carlet2014, carlet-guilley2015}. 

A pair of $[n,k,d]$-linear codes  $(\C, \cD)$ is {\it equivalent} if there is a vector $\textbf{x} \in (\fq^*)^n$ such that 
\[
\cD = \{ \textbf{x} \cdot \textbf{c} \mid \textbf{c} \in \C\},
\]
 where $\cdot$ is the coordinate-wise product of two vectors. We notice the parameters of two equivalent codes are the same. A code $\C$ is iso-dual if it is equivalent to its dual.

 The class of LCD codes has been known to be asymptotically good since they were first introduced by Massey in \cite{massey1992}. In 2004, Sendrier showed that LCD codes attain the Gilbert-Varshamov bound using the hull dimension spectra of linear codes \cite{Sendrier2004}. In 2018, L. Jin and C. Xing \cite{JinXing2018} showed that LCD algebraic geometry codes (AG codes) over $\F_{q}$ (for $q \ge 128$ even) exceed the Gilbert-Varshamov bound. In the same year, Carlet et al., \cite{CMTQP2018}, showed a more general result: for $q >3$, all linear code over $\F_q$ is equivalent to an LCD code.  These results provided a general existential guarantee of sequences of LCD codes attaining the TVZ-bound for $q>3$, but do not give us an explicit and precise way of how to construct such codes. 
 
  The study of asymptotically good AG codes goes beyond LCD codes and covers various classes, such as self-orthogonal, self-dual and block-transitive codes \cite{stich2006transitive, bassa2019self, Chara2020}. For instance, Stichtenoth and Bassa in \cite{bassa2019self} use families of Galois towers of algebraic function fields over nonprime fields $\F_q$ with $q \ge 64$ and $q \neq 125$ to guarantee the existence of infinites families of self-dual codes asymptotically which are better than the asymptotic Gilbert–Varshamov bound. In \cite{stich2006transitive}, Stichtenoth proved the existence of asymptotically good sequences of iso-dual and self-dual AG codes over $\F_q$ (with $q=\ell^2$). In 2025, in \cite{CPTQ2025} proposed a method to construct good sequences of iso-dual AG codes
  
 Recently, Bhowmick, Dalai, and Mesnager \cite{BDM2024} provided some conditions for a pair of AG codes to be an LCP of codes. As in the case of LCD codes, the conditions are not easily verified for an arbitrary curve. In \cite{CMQ25}, Castellanos, Marques, and Quoos constructed many LCPs of AG codes and LCD AG codes arising from Kummer extensions of the rational function field, hyperelliptic and elliptic function fields over finite fields. 
    
Based on the ideas in \cite{CMQ25}, we present explicit constructions of infinite sequences $(\C_i)_{i \ge 1}$ of LCD codes, and sequences $(\C_i',\cD_i')_{i \ge 1}$ of LCPs of codes over a finite field $\F_{q^2}$. Central to our approach is a certain Garcia-Stichtenoth tower of function fields over $\F_{q^2}$, where we describe suitable non-special divisors on each level of the tower. Furthermore, we provide such sequences so that all three  sequences of codes $(\C_i)_{i \ge 1}$, $(\C_i')_{i \ge 1}$ and $(\cD'_i)_{i \ge 1}$ achieve the Tsfasman-Vlăduţ-Zink bound. Remarkably, for sufficiently large $q$, these sequences exceed the classic Gilbert-Varshamov bound  representing a substantial advance in the explicit development of such codes, although their existence is guaranteed. Beyond these results, we also explicitly construct infinite sequences of self-orthogonal codes and self-dual codes attaining the Tsfasman-Vlăduţ-Zink bound from the Garcia-Stichtenoth tower.

This paper is organized as follows. In Section \ref{sec:prelim}, we introduce the notation used in this paper. We also recall some necessary background about function fields, linear and algebraic
geometry codes, some results on LCPs of codes and, finally, towers of function fields with a focus on one tower introduced by  Garcia and Stichtenoth. In Section \ref{sec:LCD_LCP_in_Tm}, we explicitly exhibit non-special divisors of degree $g_m-1$ at each step of the tower; see Proposition \ref{prop:nonspecialTmGeral}. After, we present effective constructions of LCD AG codes and LCPs of AG codes. These constructions provide  asymptotically good sequences of LCD codes over $\F_{q^2}$ attaining the TVZ bound, and sequences $(\C_i)_{i \ge 1}$ and $(\cD_i)_{i \ge 1}$ such that both meet the TVZ-bound and $(\C_i,\cD_i)$ is an LCP of codes for all $i \ge 1$; see Theorems \ref{LCDinTm} and  \ref{LCPinTm}. Finally, in Section \ref{sec:SeqSelfOrtho/Dual_Tm}, we describe self-orthogonal codes over $\F_{q^2}$ from the Garcia-Stichtenoth tower, and obtain sequences of self-orthogonal codes over $\F_{q^2}$ attaining the TVZ-bound for $q \ge 7$; see Theorem \ref{thm:Self-Orthogonal_in_Tm}. In the case $q$ is even, we can describe sequences of self-dual codes over $\F_{q^2}$ better than Gilbert-Varshamov bound for $q \ge 8$; see Theorem \ref{thm:self-dual_in_Tm}.    
\section{Preliminaries}
\label{sec:prelim}

\subsection{Algebraic Function Fields}
We first collect some concepts of function fields, algebraic geometry codes, and results that will be useful in the constructions of codes. We follow the notations as in \cite{STICH2008}.
    
Let $F/\fq$ be an algebraic function field of genus $g$ over $\F_q$. We denote by $\PP_F$ the set of places of $F/\F_q$ and by $\Div(F)$ the divisor group of $F$, that is, the free abelian group generated by the places in $F$.  Given a divisor $D=\sum_{P \in \PP_F} m_PP$, where $m_P=0$ for almost all $P$, the support of $D$ is $supp(D)=\{ P \in \PP_F \mid m_P \neq 0\}$. The group of divisors also admit a partial order $\leq$, we say that $$D=\sum_{P \in \PP_F} m_PP \leq G=\sum_{P \in \PP_F} n_PP$$ if $m_P \leq n_P$ for all $ P \in \PP_F$.
For a function  $z \in F$, $(z)$, $(z)_0$, and $(z)_\infty$ stand for the principal divisor, zero divisor, and pole divisor of $z$, respectively. More specifically, $$(z) =(z)_0-(z)_\infty = \sum_{P \in \PP_F} v_P(x),$$ where $v_P$ denotes the discrete valuation corresponding to the place $P$.

For a divisor $G \in \Div(F)$, the \textit{Riemann-Roch space} $\LL(G)$ associated to $G$ is given by
$$
\LL(G) = \{ z \in F \mid v_P(z) \ge -G \} \cup \{0\},
$$
and its dimension is defined by $\ell(G):=\dim_{\F_q} \LL(G)$.
Let 
$$
\Delta_F := \{ f dx \mid f \in F \}
$$
be the \textit{space of differentials} of $F$, where $v_P(x)$ is coprime with $q$ for some $P \in \PP_F$. For a place $P \in \PP_F$ and $t \in F$ a local parameter at $P$, we define $v_P(f dt)=v_P(f)$. The divisor associated  to a differential $\omega \in \Delta_F$ is  $(\omega)=\sum_{P \in \PP_F} v_P(\omega)P$. A divisor $W \in \Div(F)$ is called \textit{canonical} if $W=(\omega)$ for some differential $\omega \in \Delta_F$.

The Riemann-Roch Theorem tells us that 
$$
\ell(G)=\deg G + 1 - g + \ell(W - G),
$$
for all $W \in \Div(F)$ canonical divisor of $F$.
For a divisor $A \in \Div(F)$, we define the $\F_q$-vector space
$$
\Omega_F(A) := \{ \omega \in \Delta_F \mid (\omega) \ge  A \} \cup \{0\},
$$
and its dimension is denoted by $i(A) := \dim_{\F_q} \Omega_F(A)$.

\begin{definition}
    A divisor $A \in \Div(F)$ is called non-special if $i(A)=0$; otherwise, $A$ is called special. 
\end{definition}

As a consequence of the Riemann-Roch Theorem, a divisor $A \in \Div(F)$ with $\deg(A) > 2g-2$ is non-special, and the least degree of a non-special divisor is $g-1$.
For more on non-special divisors of small degree see the recent survey \cite{BK2025}.

\subsection{Algebraic Geometry Codes}
 Let $F/\fq$ be a function field. Consider $P_1, \dots, P_n \in \PP_F$ pairwise distinct rational places in $F$ and define the divisor $D := P_1 + \dots + P_n$. Let $G \in \Div(F)$ be any other divisor such that $\Supp(G) \cap \Supp(D) = \emptyset$.

\begin{definition}
    Let $\emph{ev}_D$ be the evaluation map defined by
    \begin{align*}
        \emph{ev}_D \colon \LL(G) & \longrightarrow \F_q^n \\
            z \  & \longmapsto (z(P_1), \dots, z(P_n)).
    \end{align*}
    The algebraic geometry (AG) code $\CL(D, G)$ associated with the divisors $D$ and $G$ is defined as
    $$
    \CL(D,G) := \{ \emph{ev}_D(z) \mid z \in \LL(G) \}.
    $$
\end{definition}

The following proposition gives bounds on the  parameters of an AG code.

\begin{proposition}\cite[Theorem 2.2.2]{STICH2008}
\label{parametersAGcodes}
    Let $D$ and $G$ be as above. Then, $\CL(D,G)$ is an $[n,k, d]$-code, where 
    $$
    k = \ell(G) - \ell(G - D) \quad \mbox{and} \quad d \ge n - \deg(G).
    $$
    If, in addition, $2g-2 < \deg(G) < n$, then $k= \ell(G) = \deg(G) + 1 - g$.
\end{proposition}

The dual of an AG code $\CL(D, G)$ is also an AG code, and it is given as follows. 

	\begin{proposition}\cite[Theorem 2.2.2, Proposition 8.1.2]{STICH2008}
		\label{prop:parametersAGcodes}
		Let $\CL(D,G)$  be an $[n,k, d]$-code, where $D := P_1 + \dots + P_n$. If $t$ is an element of $F$ such that $v_{P_i}(t)=1$ for all $t=1, \dots, n$,  then,
		$$
  \CL(D,G)^\perp  = \CL(D, G^\perp) \quad \mbox{with} \quad G^\perp:= D-G+(dt)-(t).
  $$
  Moreover, the code $\CL(D, G^\perp)$ has parameters $[n, n-k, d^\perp]$, where  $d^\perp \ge \deg(G) - 2g - 2$. 
	\end{proposition}

For two divisors $G, H \in \Div(F)$, we define the \textit{greatest common divisor} of $G$ and $H$ by
	$$
 \gcd(G, H) := \sum_{P \in \PP_F} \min\{ v_P(G), v_P(H) \} P,
 $$
and the \textit{least multiple divisor} of $G$ and $H$  by
$$
\lmd(G,H) := \sum_{P \in \PP_F} \max\{ v_P(G), v_P(H) \} P.
$$
We notice we have the following relation of divisors
\[
\gcd(G, H) \le G, \, H \le \lmd(G, H).
\]These divisors satisfy the following straightforward properties
\begin{equation}
    \label{BasicPropsGCDeLMD}
    \LL(G) \cap \LL(H) = \LL(\gcd(G,H)) \quad \mbox{and} \quad \gcd(G,H)+\lmd(G,H) = G+H.
\end{equation}

Let $\CL(D, G)$ and $\CL(D, H)$  be two AG codes. Since $\Supp(G)$ and $\Supp(H)$ are both disjoint from $\Supp(D)$, we have that both $\Supp(\gcd(G, H))$ and $\Supp(\lmd(G, H))$ are also disjoint from $\Supp(D)$. Thus, we can define the codes $\CL(D, \gcd(G, H))$ and $\CL(D, \lmd(G, H))$. 
In particular, we have the following lemma.

\begin{lemma}
    If $\CL(D, G)$ and $\CL(D, H)$ are two algebraic geometry codes, then
    \begin{enumerate}[label=(\roman*)]
        \item $\CL(D, \gcd(G, H)) \subseteq \CL(D, G) \cap \CL(D, H)$.
        \item  $\CL(D, G)+\CL(D, H) \subseteq \CL(D, \lmd(G, H))$.
    \end{enumerate}
\end{lemma}

In \cite{BDM2024}, Bhowmick, Dalai, and  Mesnager gave some conditions on the divisors $G$ and $H$ to describe LCPs AG codes. 
\begin{theorem}\cite[Theorem 3.5]{BDM2024} \label{thm:thmlcp3.5}
Let $\CL(D, G)$ and $\CL(D, H)$ be two algebraic geometry codes over a function field $F/\F_q$ of genus $g \neq 0$. Suppose $D$ has degree $n$, and the divisors $G$ and $H$ satisfy
\begin{enumerate}[label=(\roman*)]
\item $2g-2 < \deg(G), \, \deg(H) < n$,
\item $\ell(G)+\ell(H) = n$,
\item $\deg(\gcd(G, H))=g-1$, and 
\item both divisors $\gcd(G,H)$ and $ \lmd(G,H)- D$ are non-special.
\end{enumerate}
Then, $\left( \CL(D,G), \, \CL(D,H) \right)$ is an LCP of codes over $\F_q$.
\end{theorem}

In the case where $G$ and $H$ satisfy the conditions of Theorem \ref{thm:thmlcp3.5} and $H=D-G+(\omega)$ for a canonical divisor $(\omega)$ of $F$ with $v_{P}(\omega)=-1$ and $\res_{P}(\omega)= 1$ for all $P \in \Supp(D)$, we have  $\CL(D, G)^\perp=\CL(D, H)$. That is,  $\CL(D, G)$ is an LCD code over $\F_q$. Many constructions of LCD codes had already been shown in \cite{MTQ2018} exploring these ideas.

\subsection{Tower of Function Fields}
A tower of function fields over $\F_q$ is an infinite sequence $\cF = (F_1, F_2, F_3, \dots )$ of function fields $F_i$ over $\F_q$ such that
\begin{enumerate}[label=(\roman*)]
    \item $F_1 \subseteq F_2 \subseteq F_3 \subseteq \cdots$, and all extension $F_{i+1}/F_i$ are separable of degree $[F_{i+1} : F_i] > 1$;
    \item $\F_q$ is the full constant field of $F_i$, for all $i \ge 1$;
    \item the genus $g(F_i)$ tends to infinity as $i \to \infty$.
\end{enumerate}
  
   For a tower $\cF=(F_1, F_2, F_3 \dots)$ over $\F_q$, the \textit{limit} of $\cF$ is defined as 
$$
\lambda(\cF) = \lim_{i \to \infty} N(F_i)/g(F_i).
$$
This limit always exists (see \cite[Lemma 7.2.3]{STICH2008}) and satisfies 
\[
 0 \le  \lambda(\cF) \le A(q).
\]
Thus, from Drinfeld-Vlăduţ, we have
$$
0 \le \lambda(\cF) \le A(q) \le q^{1/2} -1.
$$
A tower $\cF$ is called \textit{asymptotically good} over $\F_q$ if 
$\lambda(\cF) > 0$. Otherwise, the tower is called \textit{asymptotically bad}. We say that $\cF$ is asymptotically optimal if $\lambda(\cF)=A(q)$.

It is worth to point out the area of towers of function fields has been a very fruitful area of research in the last decades as can be seen in the survey by Beelen, see \cite{B2025}.

\section{A  Tower from Garcia and Stichtenoth}
\label{SubsectionGSTower} In this section, we introduce a tower of function fields with many rational places. This tower was introduced by A. Garcia and H. Stichtenoth  in \cite{GStower1996} and will be used to construct sequences of LCD AG codes and LCPs of AG codes exceeding the Gilbert-Varshamov bound for a large $q$. 

 The recursive tower given in  \cite{GStower1996} is as follows.
Let $\cT=(T_1, T_2, T_3, \dots )$ be the tower of function fields over $\F_{q^2}$ defined by 
\begin{equation}\label{eq:GStower2}
\begin{cases} 
    T_1 := \F_{q^2}(x_1) \text{ and for } m \ge 2, \\
    T_m:= T_{m-1}(x_m)  \mbox{ where }  x_{i+1}^q+x_{i+1} = \frac{ x_i^q }{ x_{i}^{q-1} + 1 }, \mbox{ for } i=1, \dots, m-1. 
\end{cases}
\end{equation}
 In the same paper, they compute the genus $g_m:=g(T_m)$ of the function field $T_m/\F_{q^2}$, which is given by 
\begin{equation}
    \label{genusTm}
    g_m =  \begin{cases}
        (q^{m/2}-1)^2, & \quad \mbox{if} \quad m \equiv 0 \pmod{2}, \\
        (q^{(m+1)/2}-1)(q^{(m-1)/2}-1), & \quad \mbox{if} \quad m \equiv 1 \pmod{2}.
    \end{cases}
\end{equation}
Furthermore, they have also shown that $\cT$ attains the Drinfeld-Vlăduţ bound over $\F_{q^2}$, i.e.,
$$
\lambda(\cT)=q-1.
$$

For our constructions, we need to know the behaviour of the rational places of $T_1$ in the extension $T_m/T_1$ for $m \ge 2$.
Consider the two sets 
$$
\Omega := \{ \alpha \in \F_{q^2} \mid \alpha^q + \alpha = 0 \} \quad \mbox{and} \quad \Omega^* := \Omega \setminus \{ 0 \}.
$$

Let $P_\infty$ be the unique pole of $x_1$ in $T_1=\F_{q^2}(x_1)$, and for $a \in \fqs$,  let $P_a$ be the zero of $x_1-a$ in $T_1$.
The ramified places of $T_m$ lie above either $P_\infty$ or  $P_\alpha$ with $\alpha \in \Omega$. The places $P_\infty$ and $P_\alpha$, with $\alpha \in \Omega^*$, are totally ramified in the extension $T_m/T_1$. We denote $P_\infty^{(m)}$ (resp., $P_{\alpha}^{(m)}$ for $\alpha \in \Omega^*$) the unique place in $T_m$ lying above $P_\infty$ (resp., $P_{\alpha}$ for $\alpha \in \Omega^*$). The other places $P_a \in \PP_{T_1}$ with $a \in \F_{q^2} \setminus \Omega$ split completely in the extension $T_m/T_1$, see \cite[Lemma 3.9]{GStower1996}.

We denote by $Q_k^{(k
)}$ the unique place of $T_k$ which is the only common zero of $x_1, \dots, x_k$ this place lies over  $P_0=Q_1^{(1)}$. 
The following lemma tells us how the decomposition of the zero $P_0^{(1)}$ of $x_1$ behaves in the extension $T_m/T_1$, see Figure \ref{fig:RamificationTm}.

\begin{lemma}\cite[Lemma 3.5]{GStower1996}
\label{behaviourzeroinTm}
    Let $1 \le k < m$ and $Q \in \PP_{T_m}$ be a place having the following properties:
    \begin{enumerate}
        \item $Q$ is a common zero of $x_1, \dots, x_k$;
        \item $Q$ is a zero of $x_{k+1}-\alpha$, for some $\alpha \in \Omega^*$;
        \item $Q$ is a common pole of $x_{k+2}, \dots, x_m$.
    \end{enumerate} 
    Then one has:
    \begin{enumerate}[label=(\roman*)]
        \item If $k \ge (m-1)/2$, then the place $Q$ is unramified in $T_m/T_{m-1}$.
        \item For $k < (m-1)/2$, the place $Q$ is totally ramified in $T_m/T_{2k+1}$; and for $2k+1 \le s \le m$, the restriction of $Q$ to $T_s$ is unramified in the extension $T_s/\F_{q^2}(x_s)$.
        \item If $k < (m-1)/2$, the different exponent $d(Q)$ of $Q$ in the extension $T_m/T_{m-1}$ is given by $d(Q)=2(q-1)$.
    \end{enumerate}
\end{lemma}

\input{ramificationdiagram}

In \cite{PELLIKAAN1998}, Pellikaan, Stichtenoth and Torres showed that the Weierstrass semigroup $H(P_\infty^{(m)})$ at the unique pole of $x_1$ is given by 
$$\begin{cases}
H(P_\infty^{(1)})=\mathbb{N}_0, \text{ and }\\
 H(P_\infty^{(m+1)})=q \cdot H(P_\infty^{(m)}) \cup \{n \in \mathbb{N}_0 \mid n \ge c_m \},
\end{cases}
$$
where 
\[
c_m :=q^m - q^{ \ceil*{m/2} } = 
\begin{cases}
    q^m - q^{m/2}, & \mbox{ if } m \equiv 0 \pmod{2}. \\
    q^m - q^{(m+1)/2}, & \mbox{ if } m \equiv 1 \pmod{2}.
\end{cases}
\]
One of the ingredients used by the authors in \cite{PELLIKAAN1998} to obtain $H(P_\infty^{(m)})$ was introducing a function $\pi_j$ and showing that it satisfies specific desired properties. Namely,
for $1 \le j \le m$, let 
\[
\pi_j := \prod_{i=1}^j (x_i^{q-1}+1).
\]

The function $\pi_j \in T_m$ satisfies the following properties.
\begin{theorem} \cite[Lemmas 3.4 and 3.9]{PELLIKAAN1998} \label{thm:PSTresults}
    For $1 \le j \le m$, 
    \begin{enumerate}
   \item   The principal divisor of $\pi_j$ over $T_m$ is given by 
    \[
    (\pi_j)^{T_m} = (\pi_j)_0^{T_m}-( q^m-q^{m-j})P_\infty^{(m)},
    \]
    where $\Supp( (\pi_j)_0^{T_m} ) = \{P \in \PP_{T_m} \mid P \mbox{ is a zero of } x_i^{q-1}+1 \mbox{ for some } 1 \le i \le j  \}$.
    \item  Define $A^{(1)} := 0$ and, for $m \ge 2$, 
    $
    A^{(m)} := \sum_{P \in \mathcal{Z}_{\floor*{m/2}}^{(m)}} P. 
    $
    Then 
    \begin{enumerate}[label=(\roman*)]
        \item $\deg(c_mP_\infty^{(m)}-A^{(m)})=g_m$, and
        \item $\ell( c_mP_\infty^{(m)}-A^{(m)} ) = 1$.
    \end{enumerate}
    \end{enumerate}
\end{theorem}

\section{Non-special Divisors of Degree \texorpdfstring{$g_m-1$}{gm-1} on \texorpdfstring{$T_m$}{Tm}} \label{sec:NonSpecial_Tm}

In this section we are going determine non-special divisors of small degree in the tower $\cT$ given by Equation \ref{eq:GStower2}. This divisors will play a fundamental role on the construction of  LCP of codes and LCD codes using Theorem  \ref{thm:thmlcp3.5}. These ideas where already explored in the paper \cite{CMQ25}.

Consider the tower $\cT$ given by Equation \ref{eq:GStower2} and 
$m \ge 2$.

At first we describe non-special divisors of degree $g_m-1$ at each level $T_m$ of the tower $\cT$.
For $0  \le k \le m-1$,  define
$$
X_k^{(m)} := \{  Q \in \PP_{T_m} \mid Q \mbox{ is a zero of } x_{k+1}^{q-1}+1 \} 
$$
and 
$$
 A_{k}^{(m)} := \sum_{Q \in X_{k}^{(m)} } Q.
$$
For $\alpha \in \Omega^*$ and $0  \le k \le m-1$, we set
	\begin{equation}
	    \label{eq:DefXka}
        	X_{k,\alpha}^{(m)} := \{ Q \in \PP_{T_m} \mid Q \mbox{ is a zero of } x_{k+1}-\alpha \} 
	\end{equation} 
	and
	\begin{equation}
	   \label{eq:DefAka}
         A_{k,\alpha}^{(m)} := \sum_{Q \in X_{k,\alpha}^{(m)} } Q. 
	\end{equation}
Note that 
\[
X_k^{(m)} = \bigcup_{\alpha \in \Omega^*} X_{k, \alpha}^{(m)} \quad \mbox{and} \quad A_k^{(m)} = \sum_{\alpha \in \Omega^*} A_{k, \alpha}^{(m)}.
\]

\begin{lemma}
    \label{lemma-degA_k}
    Let  $0 \le k \le  \floor*{\frac{m-2}{2}}$. Then the degree of the divisor $A_{k}^{(m)} $ is given by
    $$
    \deg(A_{k}^{(m)})=q^{k+1}-q^k.
    $$
\end{lemma}
\begin{proof}
     For $k=0$, we have that $A_0^{(m)}= \sum_{\alpha \in \Omega^*} P_\alpha^{(m)}$ with $\deg(P_\alpha^{(m)})=1$ for all $\alpha \in \Omega^*$, so $\deg(A_0^{(m)})=q-1$. Now, let $1 \le k \le\floor*{(m-2)/2}$. From \cite[Lemma 3.6]{GStower1996}, we have $\deg(A_{k, \alpha}^{(m)}) = q^k$ for any $\alpha \in \Omega^*$. Since $A_k^{(m)}=\sum_{ \alpha \in \Omega^* } A_{k, \alpha}^{(m)}$, we obtain
         \[
         \deg(A_k^{(m)}) = \sum_{ \alpha \in \Omega^* } \deg(A_{k, \alpha}^{(m)}) = q^k(q-1).
         \]
\end{proof}

\begin{proposition}
        \label{prop:nonspecialTmGeral}
		The divisors
		\begin{equation}
		    \label{eq:nonspecialdivTm1}
            \sum_{\alpha \in \Omega^*} (q^{\ceil*{m/2}}-1)P_{\alpha}^{(m)} +   \sum_{k=1}^{\floor*{m/2}-1} (q^{\ceil*{m/2}}-1) A_k^{(m)}- P_\infty^{(m)}
		\end{equation}
	and
	\begin{equation}
		    \label{eq:nonspecialdivTm2}
	\sum_{\alpha \in \Omega^*} (q^{m-1}-1)P_\alpha^{(m)} - \sum_{k=1}^{ \floor*{(m-2)/2} } A_{k}^{(m)} + (q^{m-1} - q^{\ceil*{m/2}}-1)P_\infty^{(m)}
	\end{equation}
	of $T_m$ are both non-special of degree $g_m-1$.
\end{proposition}
\begin{proof}
	Consider the function $\pi_{\floor{m/2}}=\prod_{i=1}^{\floor{m/2}} (x_i^{q-1}+1) \in T_m$. From Theorem \ref{thm:PSTresults} (i), we have 
    \[
    (\pi_{\floor{m/2}})^{T_m} = (\pi_{\floor{m/2}})_0^{T_m} - (q^m - q^{m - \floor{m/2}})P_\infty^{m}.
    \]
    Moreover, the zeros of $\pi_{\floor{m/2}}$ belong to $X_k^{(m)}$ for $0 \le k \le \floor{m/2}-1$, and the valuation of $\pi_{\floor{m/2}}$ at $Q \in X_k^{(m)}$ is given by
    \[
    v_Q(\pi_{\floor{m/2}}) = q^{m-(\floor{m/2}-1)-1} = q^{m-\floor{m/2}},
    \]
    see \cite[Appendix A.1.(a)]{SAKSD2001}.
	Thus,
	\begin{align*}
		\left( \pi_{\floor*{m/2}} \right)^{T_m} & = (\pi_{\floor{m/2}})_0^{T_m} - (q^m - q^{m - \floor{m/2}})P_\infty^{m} \\
		& = \sum_{k=0}^{ \floor*{m/2} - 1 } q^{m- \floor{m/2}} A_k^{(m)}  - \left( q^m - q^{\ceil*{m/2}} \right) P_\infty^{(m)} \\
		& = \sum_{k=0}^{ \floor*{(m-2)/2}  } q^{\ceil*{m/2}} A_k^{(m)}  - \left( q^m - q^{\ceil*{m/2}} \right) P_\infty^{(m)}. 
	\end{align*}
	Hence, the zero divisor $(\pi_{\floor*{m/2}})_0^{T_m}$ of the function $\pi_{\floor*{m/2}}$ in $T_m$ is given by 
	\[
	\left( \pi_{\floor*{m/2}} \right)_0^{T_m}= \sum_{k=0}^{ \floor*{(m-2)/2}  } q^{\ceil*{m/2}} A_k^{(m)} .
	\]
    The divisor
    $\left( q^m - q^{\ceil*{m/2}} \right)P_\infty^{(m)}-A^{(m)} $
    is  non-special divisor of degree $g_m$ by Theorem \ref{thm:PSTresults} (ii).  From
    \begin{equation}
    \label{eq:A=sum A_{a, k}}
          A^{(m)}=\sum_{k=0}^{ \floor*{ m/2 } - 1} A_k^{(m)}, 
    \end{equation}
    we have that 
    \begin{align*}  
    \left( q^m - q^{\ceil*{m/2}} \right)P_\infty^{(m)}-A^{(m)} + \left( \pi_{\floor*{m/2}} \right)^{T_m} & = \left( \pi_{\floor*{m/2}} \right)_0^{T_m} - \sum_{k=0}^{ \floor*{ m/2 } - 1} A_k^{(m)} 
    \\ & = \sum_{k=0}^{ \floor*{(m-2)/2}  } \left( q^{\ceil*{m/2}} - 1 \right) A_k^{(m)}  
   \end{align*}
   is also a non-special of degree $g_m$. 
   
   Moreover, $P_\infty^{(m)}$ is a rational place of $T_m$ such that 
   $$P_\infty^{(m)} \notin \Supp \left( \sum_{k=0}^{ \floor*{(m-2)/2}  } \left( q^{\ceil*{m/2}} - 1 \right) A_k^{(m)}   \right).$$ 
   So, from \cite[Lemma 3]{BL2006}, 
   
    \[
    \sum_{k=0}^{ \floor*{(m-2)/2}  } \left( q^{\ceil*{m/2}} - 1 \right) A_k^{(m)} - P_\infty^{(m)} =  \sum_{\alpha \in \Omega^*} \left( q^{\ceil*{m/2}}-1 \right)P_\alpha^{(m)} +  \sum_{k=1}^{ \floor*{(m-2)/2}  } (q^{\ceil*{m/2}}-1) A_k - P_\infty^{(m)}
    \]
    is a non-special divisor of degree $g_m-1$. We conclude the divisor given in (\ref{eq:nonspecialdivTm1}) is non-special of degree $g_m-1$.
    
	Now, let $\rho := \pi_{\floor*{m/2}} \cdot (x_1^{q-1}+1)^{-1} \in T_m$. We have that
	\begin{align*}
    (\rho )^{T_m} & = 
		\left(\pi_{\floor*{m/2}} \cdot (x_1^{q-1}+1)^{-1}  \right)^{T_m} \\
        & = \left( \pi_{\floor*{m/2}} \right)^{T_m} - (x_1^{q-1}+1)^{T_m} \\
		& = \left( \pi_{\floor*{m/2}} \right)_0^{T_m} - \left( \pi_{\floor*{m/2}} \right)_\infty^{T_m} - (x_1^{q-1}+1)^{T_m},
	\end{align*}
	this yields
	\begin{align*}
		\left( \pi_{\floor*{m/2}} \right)_0^{T_m} - A^{(m)} - P_\infty^{(m)} -  \left( \rho \right)^{T_m} & = \left( \pi_{\floor*{m/2}} \right)_\infty^{T_m} - A^{(m)} - P_\infty^{(m)} + (x_1^{q-1}+1)^{T_m}
		\\
		& = 
		\sum_{\alpha \in \Omega^*} q^{m-1}P_\alpha^{(m)} + \left(q^{m-1} - q^{\ceil*{m/2}} - 1 \right) P_\infty^{(m)} - A^{(m)}.
	\end{align*} 
	From above and Equation (\ref{eq:A=sum A_{a, k}}), we conclude that the divisor 
	\[
	\sum_{\alpha \in \Omega^*} (q^{m-1}-1)P_\alpha^{(m)} - \sum_{k=1}^{\floor*{(m-2)/2}} A_k^{(m)} + \left(q^{m-1} - q^{\ceil*{m/2}} - 1 \right) P_\infty^{(m)}
	\]
	is also a non-special divisor in $T_m$ degree $g_m-1$, since it is equivalent to the non-special divisor $\left( \pi_{\floor*{m/2}} \right)_0^{T_m} - A^{(m)} - P_\infty^{(m)}$ of degree $g_m-1$.
\end{proof}


\section{ Sequences of LCD AG codes and LCP of AG Codes attaining TVZ-bound}
\label{sec:LCD_LCP_in_Tm}

In this section we construct sequences of LCD AG codes and LCPs of AG codes over $\F_{q^2}$ attaining the Tsfasman–Vlăduţ–Zink bound using the tower $\cT$  presented in Section \ref{SubsectionGSTower}.  

We begin by defining AG codes from each $T_m$ for $m \ge 2$.

 \begin{definition}
	Let $t := (x_1^{q^2} - x_1)/(x_1^q+x_1) \in T_1 \subseteq T_m$. Define the divisor $D^{(m)} \in \Div(T_m)$ as
	\[
	D^{(m)} :=  (t)_0^{T_m}
	=\left( \frac{x_1^{q^2} - x_1}{x_1^q+x_1} \right)_0^{T_m} ,		\]
	and $n_m := \deg(D^{(m)})$ its degree.
\end{definition}
The zeros of $t$ in $T_1$ are completely split in the extension $T_m/T_1$. In this way, the divisor $D^{(m)}$ is a sum of 
\[
n_m = \deg(D^{(m)}) = [T_m : T_1](q^2-q) = q^m(q-1).
\]
pairwise distinct rational places. 
Furthermore, observe that in $F_1$
\begin{equation*}
	dt = -(x_1^q+x_1)^{q-2} dx_1
\end{equation*}
and so, by \cite[Remark 4.3.7, (c)]{STICH2008},
\begin{equation}
	\label{Div-dt_differential}
	(dt)^{T_m}= (q-2)(x_1^{q}+x_1)^{T_m} -2(x_1)^{T_m}_\infty + \Diff(T_m/T_1),
\end{equation}
where $\Diff(T_m/T_1)$ denotes the different divisor in the extension $T_m/T_1$.

Choosing $G^{(m)} \in \Div(T_m)$  another divisor of $T_m$ such that 
\[
\Supp(G^{(m)}) \cap \Supp(D^{(m)}) = \emptyset \quad \mbox{and} \quad  2g_m-2 <  \deg(G^{(m)}) < n_m, 
\] 
we have $\CL(D^{(m)}, G^{(m)})$ is an $[n_m, \, k_m, d_m]$-code over $\F_{q^2}$ with
\begin{equation}
	\label{ParametersAGCodesTm}
	n_m=q^{m+1}-q^m, \qquad k_m = \ell(G^{(m)}) \quad \mbox{and} \quad d_m \ge q^{m+1}-q^m - \deg(G^{(m)}),
\end{equation}
by Proposition \ref{parametersAGcodes}.
Consider the differential $\omega^{(m)} := dt/t$ of $T_m$ and $W^{(m)}:=(\omega^{(m)})^{T_m}$ its canonical divisor on $T_m$. By Proposition \ref{prop:parametersAGcodes},  since $D^{(m)} =  (t)_0^{T_m}$, the dual code of  $\CL(D^{(m)}, G^{(m)})$ is given by
\begin{equation}
	\label{eq:DualAGCodeinTm}
	\CL(D^{(m)}, G^{(m)})^\perp = \CL(D^{(m)}, D^{(m)} - G^{(m)} + W^{(m)}).
\end{equation}

Our goal is to  find explicit divisors $G^{(m)}$ and $H^{(m)}$  at each level $T_m$ of the tower $\cT $ to construct  LCD AG codes $\CL(D^{(m)}, G^{(m)})$ over $\F_{q^2}$, and describe code pairs 
$$
\left( \CL(D^{(m)}, G^{(m)}), \CL(D^{(m)}, H^{(m)})\right)
$$ 
that are LCP of AG codes over $\F_{q^2}$.

For the construction of LCD codes from $T_m$, we initially need to describe the dual of an AG code $\CL(D^{(m)}, G^{(m)})$ which is given by Equation \ref{eq:DualAGCodeinTm}. For this reason, in the following lemmas, we  describe the different divisor $\Diff(T_m/T_1)$ of the extension $T_m/T_1$, and the canonical divisor $W^{(m)}=(dt/t)^{T_m}$. 

\begin{lemma} \label{DifferentDivisorTm}
For $m \ge 2$, the different divisor $\Diff(T_m/T_1)$ of the extension $T_m/T_1$ is given by
		$$
		\Diff(T_m/T_1) = \sum_{\alpha \in \Omega^*} 2(q^{m-1}-1)P_\alpha^{(m)} +2(q^{m-1}-1)P_\infty^{(m)} +  \sum_{k=1}^{\floor*{(m-2)/2}}  2(q^{m-2k-1} - 1)A_{k}^{(m)}.
		$$
	\end{lemma}
		\begin{proof} The only ramified places in $T_m$ are $P_\alpha^{(m)}$ for $\alpha \in \Omega^*$, $P_\infty^{(m)}$, and the places in $A_k$ for $1 \le k \le \floor*{(m-2)/2} = \floor*{m/2}-1$.
			By Lemma \ref{behaviourzeroinTm} and the transitivity of the different exponent (see \cite[Corollary 3.4.11]{STICH2008}, and \cite[Section 4]{ADKS2001}) for $1 \le k \le \floor*{(m-2)/2}$ and $Q \in A_k^{(m)}$, we have
			$$ 
			d(P_\alpha^{(m)}) = d(P_\infty^{(m)}) = 2(q^{m-1}-1) \quad \mbox{and} \quad d(Q)=2(q^{m-(2k+1)} - 1). 
			$$
			in the extension $T_m/T_1$. Therefore, we obtain that 
            \[
			\Diff(T_m/T_1) = \sum_{\alpha \in \Omega^*} 2(q^{m-1}-1)P_\alpha^{(m)} +2(q^{m-1}-1)P_\infty^{(m)} +  \sum_{k=1}^{\floor*{(m-2)/2}}  2(q^{m-2k-1} - 1)A_{k}^{(m)}.
            \]
		\end{proof}

	\begin{lemma}
		\label{CanonicalDivisorTm}
		The canonical divisor $W^{(m)} = (dt/t)^{T_m}$ is given by
		$$
		-D^{(m)}+(x_1^{q-2})_0 + \sum_{\alpha \in \Omega^*} (q^{m}-2)P_\alpha^{(m)} +(q^{m}-2)P_\infty^{(m)} +  \sum_{k=1}^{\floor*{(m-2)/2}} 2(q^{m-2k-1} - 1)A_{k}^{(m)}.
		$$
	\end{lemma}
	\begin{proof}
		Let $t=(x_1^{q^2}-x_1)/(x_1^q+x_1) \in T_m$. By derivation rules, the differential $dt$ is given by $$dt=-(x_1^q+x_1)^{q-2}dx_1.$$ Moreover, from \cite[Remark 4.4.7 (c)]{STICH2008} and Lemma \ref{DifferentDivisorTm}, 
		\begin{align*}
			(dx_1)^{T_m} & = -2(x_1)_\infty^{T_m}+\Diff(T_m/T_1) \\
			& = \sum_{\alpha \in \Omega^*} 2(q^{m-1}-1)P_\alpha^{(m)}  +  \sum_{k=1}^{\floor*{(m-2)/2}}  2(q^{m-2k-1} - 1)A_{k}^{(m)} -2 P_\infty^{(m)}.
		\end{align*}
		
		Hence, since 
		\begin{align*}
		&(x_1^{q-2})^{T_m}=(x_1^{q-2})_0^{T_m}-(q^m-2q^{m-1})P_\infty^{(m)}\text{, and}\\		
		&\left( (x_1^{q-1}+1)^{q-2} \right)^{T_m} = \sum_{\alpha \in \Omega^*} (q^m-2q^{m-1})P_{\alpha}^{(m)} - (q^{m+1}-3q^m+2q^{m-1})P_\infty^{(m)},
		\end{align*}		
		we obtain 
		\begin{align*}  
			(dt)^{T_m} & = \left( (x_1^q+x_1)^{q-2} \right)^{T_m} + (dx_1)^{T_m} \\
			& = (x_1^{q-2})^{T_m} + \left( (x_1^{q-1}+1)^{q-2} \right)^{T_m} + (dx_1)^{T_m} \\
			& =
            (x_1^{q-2})_0^{T_m} + \sum_{\alpha \in \Omega^*} (q^{m}-2)P_\alpha^{(m)}  +  \sum_{k=1}^{\floor*{(m-2)/2}}  2(q^{m-2k-1} - 1)A_{k}^{(m)} \\
            & \quad- (q^{m+1}-2q^m+2)P_\infty^{(m)}.
    \end{align*}
		Therefore,
		\begin{align*}
			W^{(m)} & = -(t)_0^{T_m} + (t)_\infty^{T_m} + (dt)^{T_m} \\
			& = -D^{(m)} + (q^{m+1}-q^m)P_\infty^{(m)} +(dt)^{T_m} \\
			& = -D^{(m)} + (x_1^{q-2})_0^{T_m} + \sum_{\alpha \in \Omega^*} (q^{m}-2)P_\alpha^{(m)}  +  \sum_{k=1}^{\floor*{(m-2)/2}}  2(q^{m-2k-1} - 1)A_{k}^{(m)} \\
         & \quad + (q^m-2)P_\infty^{(m)}.
         \end{align*}
	\end{proof}

With the previous lemmas, we can describe the dual of an AG code $\CL(D^{(m)}, G^{(m)})$ with $\Supp(G^{(m)})$ contained in $\Supp( (\pi_{\floor*{m/2}})^{T_m})$ in terms of the divisors $A_k^{(m)}$ for $0 \le k \le \floor*{m/2}-1$ and the unique pole $P_\infty^{(m)}$ of $x_1$ in $T_m$. As a result, let us see how to explicitly construct LCD codes over $\F_{q^2}$ from $T_m$ for all $m \ge 2$.

\begin{theorem}
	\label{LCDinTmGeral}
	Let $D^{(m)}=(t)_0^{T_m}$ be the zero divisor of the function $t:=(x_1^{q^2}-x_1)/(x_1^q+x_1) \in T_m$.
	Suppose that there are divisors $G^{(m)}, H^{(m)} \in \Div(T_m)$ satisfying the following three properties:
	\begin{enumerate}[label=(\roman*)]
		\item \label{cond:LCDinTmC1}
        $\deg(G^{(m)}), \deg(H^{(m)}) < q^{m+1}-q^m$;
		\item \label{cond:LCDinTmC2}
        $ \gcd\!\left(G^{(m)}, \, H^{(m)} \right) = \sum_{\alpha \in \Omega^*} (q^{m-1}-1)P_\alpha^{(m)} -  \sum_{k=1}^{ \floor*{(m-2)/2} } A_{k}^{(m)} + (q^{m-1} - q^{\ceil*{m/2}}-1)P_\infty^{(m)}$; 
	 \item \label{cond:LCDinTmC3} 
	 \begin{align*}
	     G^{(m)}+ H^{(m)} &= (x_1^{q-2})_0 + \sum_{\alpha \in \Omega^*} (q^{m}-2)P_\alpha^{(m)} \\
	     &+   \sum_{k=1}^{\floor*{(m-2)/2}} 2(q^{m-2k-1} - 1)A_{k}^{(m)} 
         + (q^{m}-2)P_\infty^{(m)}.
          \end{align*}
	\end{enumerate}

Then, $\CL(D^{(m)}, G^{(m)})$ is an LCD $[n_m, k_m, d_m]$-code over $\F_{q^2}$, where
$$
 \begin{cases}
	n_m = q^{m+1}-q^m,  \\	
	k_m=\deg(G^{(m)})+1-(q^{\floor*{m/2} } - 1)(q^{\ceil*{m/2} } -1) \mbox{, and} \\
	d_m \ge q^{m+1}-q^m-\deg(G^{(m)}).
	\end{cases}
	$$
    Furthermore, we have $\CL(D^{(m)}, G^{(m)})^\perp = \CL(D^{(m)}, H^{(m)})$.
   \end{theorem}
\begin{proof}
	Let $G^{(m)}, H^{(m)} \in \Div(T_m)$ satisfying the Conditions \ref{cond:LCDinTmC1}, \ref{cond:LCDinTmC2} and \ref{cond:LCDinTmC3}. 
	We are going to show that $G^{(m)}$ and $H^{(m)}$ satisfy 
    $\gcd\!\left(G^{(m)}, \, H^{(m)} \right)$ is a non-special divisor of degree $g_m-1$, and $\CL(D^{(m)}, G^{(m)})^\perp=\CL(D^{(m)}, H^{(m)})$.

Let $W^{(m)} = (dt/t)^{T_m}$. By Condition \ref{cond:LCDinTmC3} and Lemma \ref{CanonicalDivisorTm}, we have 
	\[
	W^{(m)} = G^{(m)} + H^{(m)} - D^{(m)}.
	\]
	Hence, 
    \[
    \CL(D^{(m)}, G^{(m)})^\perp = \CL(D^{(m)}, H^{(m)})
    \]
    by Equation (\ref{eq:DualAGCodeinTm}).  
    
    On the other hand, Condition \ref{cond:LCDinTmC2} and Proposition \ref{prop:nonspecialTmGeral} tell us that $\gcd\!\left(G^{(m)}, \, H^{(m)} \right)$ is a non-special divisor of degree $g_m-1$. Therefore, $\CL(D^{(m)}, G^{(m)})$ is an LCD code over $\F_{q^2}$. Lastly, since 
    \[
    G^{(m)} \ge \gcd\!\left(G^{(m)}, \, H^{(m)} \right)
    \]
    and $\gcd\!\left(G^{(m)}, \, H^{(m)} \right)$ is non-special, we have that $G^{(m)}$ is also a non-special divisor, that is, $$\ell(G^{(m)})=\deg(G^{(m)})+1-g_m.$$ Hence, by Condition \ref{cond:LCDinTmC1} and Equation (\ref{ParametersAGCodesTm}), the code $\CL(D^{(m)}, G^{(m)})$ has parameters 
	$$[q^{m+1}-q^m, \deg(G^{(m)})+1-g_m, \ge q^{m+1}-q^m-\deg(G^{(m)}) ],$$ where the genus $g_m$ of $T_m$ is given by $(q^{\floor*{m/2} } - 1)(q^{\ceil*{m/2} } -1)$.
\end{proof}

 The next theorem  show us that there are pairs $\left( G^{(m)}, \, H^{(m)} \right)$ satisfying the conditions of Theorem \ref{LCDinTmGeral} with $\CL(D^{(m)}, G^{(m)})$ being a non-trivial code over $\F_{q^2}$.

\begin{theorem}
	\label{LCDinTm}
	Let $q \geq 4$. Let $D^{(m)} = (t)_0$ be the zero divisor of $t=(x_1^{q^2}-x_1)/(x_1^q+x_1)$ in $T_m$. Consider the divisors
	\[
	G^{(m)} := (x_1^{q-2})_0 + \sum_{\alpha \in \Omega^*} (q^{m-1}-1) P_\alpha^{(m)} -  \sum_{k=1}^{\floor*{(m-2)/2}} A_{k}^{(m)} + (q^m-q^{m-1}+q^{\ceil*{m/2}}-1)P_\infty^{(m)}
	\]
	and 
	\[
	H^{(m)} :=  \sum_{\alpha \in \Omega^*} (q^{m}-q^{m-1}-1) P_\alpha^{(m)} +  \sum_{k=1}^{\floor*{(m-2)/2}} (2q^{m-2k-1} - 1)A_{k}^{(m)} + (q^{m-1}-q^{\ceil*{m/2}}-1)P_\infty^{(m)}.
	\]
	Then, $\CL(D^{(m)}, G^{(m)})$ is an LCD code over $\F_{q^2}$ with parameters 
    \[
    [q^{m+1}-q^m, \ 2q^m-4q^{m-1}+2q^{\ceil*{m/2}}, \, \ge q^{m+1} - 4q^m + 4q^{m-1} - q^{\ceil*{m/2}} + q^{\floor*{m/2}}]
    \]
     and
	\[
	\CL(D^{(m)}, G^{(m)})^\perp =  \CL(D^{(m)}, H^{(m)}).
	\]
    Moreover, the sequence of codes $(\CL(D^{(m)}, G^{(m)}))_{m \ge 2}$ attains the Tsfasman–Vlăduţ–Zink bound.
\end{theorem}
\begin{proof}
	We are going to verify that the divisors $G^{(m)}$ and $H^{(m)}$ satisfy the conditions of Theorem \ref{LCDinTmGeral}.
	
	At first, we compute the degree of $G^{(m)}$ and $H^{(m)}$. By Lemma \ref{lemma-degA_k} we have $\deg(A_{k}^{(m)})=q^{k+1}-q^k $, so 
\begin{align*}
			\deg (G^{(m)}) & =(q-2)q^{m-1} + (q-1)(q^{m-1}-1) -(q^{ \floor*{m/2} }-q)
				+ q^m-q^{m-1} + q^{\ceil*{m/2}} - 1\\
			& = 3q^m - 4q^{m-1}+ q^{\ceil*{m/2}} -q^{\floor*{m/2}}.
	\end{align*}
    and
	\begin{align*}
    	\deg (H^{(m)}) & = (q-1)(q^m - q^{m-1} - 1)  +\sum_{k=1}^{\floor*{(m-2)/2}}\big[ (2q^{m-2k-1}-1)(q^{k+1}-q^k)\big]  + q^{m-1} - q^{\ceil*{m/2}} - 1 \\
    	& =   q^{m+1} - 2q^m + 2q^{m-1} - q^{\ceil*{m/2}} - q + \sum_{k=1}^{\floor*{(m-2)/2}} \big[ 2(q^{m-k}-q^{m-k-1}) -  (q^{k+1}-q^k) \big]  \\ 
    	& =  q^{m+1} - 2q^m + 2q^{m-1} - q^{\ceil*{m/2}} -q +2(q^{m-1}-q^{m-\floor*{m/2}}) - (q^{ \floor*{m/2} } - q)\\	
    	& =  q^{m+1} - 2q^m + 4q^{m-1} - 3q^{\ceil*{m/2}} - q^{ \floor*{m/2} },
    \end{align*}
    where use that $m-\floor*{m/2}=\ceil*{m/2}$.
	Thus, for $q > 3$,
    \[
 \deg(G^{(m)}) = 3q^m - 4q^{m-1} + q^{\ceil*{m/2}} - q^{\floor*{m/2}} < q^{m+1}-q^m=\deg(D^{(m)})
 \]
 and
     \[
 \deg(H^{(m)}) = q^{m+1} - 2q^m + 4q^{m-1} - 3q^{\ceil*{m/2}} - q^{\floor*{m/2}} < q^{m+1}-q^m=\deg(D^{(m)}).
 \] 
 Moreover, it is straightforward to see that 
	\[
	\gcd\!\left(G^{(m)}, \, H^{(m)} \right)=\sum_{\alpha \in \Omega^*} (q^{m-1}-1)P_\alpha^{(m)} -  \sum_{k=1}^{ \floor*{(m-2)/2} } A_{k}^{(m)} + (q^{m-1} - q^{\ceil*{m/2}}-1)P_\infty^{(m)}
	\]
	and 
	\[
	G^{(m)}+ H^{(m)} = (x_1^{q-2})_0 + \sum_{\alpha \in \Omega^*} (q^{m}-2)P_\alpha^{(m)}+  \sum_{k=1}^{\floor*{(m-2)/2}} 2(q^{m-2k-1} - 1)A_{k}^{(m)}  +(q^{m}-2)P_\infty^{(m)}.
	\]
	Therefore, $\CL(D^{(m)}, G^{(m)})$ is an LCD code with $\CL(D^{(m)}, G^{(m)})^\perp =  \CL(D^{(m)}, H^{(m)})$ by Theorem \ref{LCDinTmGeral}. Furthermore, the rate $R_m=k_m/n_m$ and relative minimum distance $\delta_m=d_m/n_m$ of $\CL(D^{(m)}, G^{(m)})$ are given, respectively, by 
    \begin{align*}
        R_m & = \frac{k_m}{n_m}  = \frac{ 2q^m-4q^{m-1}+2q^{\ceil*{m/2}} }{ q^{m+1} - q^m } \longrightarrow \frac{2q-4}{q^2-q}  \qquad \mbox{as } m \to \infty 
    \end{align*}
    and
    \begin{align*}
        \delta_m & = \frac{d_m}{n_m}   \ge \frac{q^{m+1} - 4q^m + 4q^{m-1} - q^{\ceil*{m/2}} + q^{\floor*{m/2}}}{q^{m+1} - q^m} \longrightarrow \frac{q^2-4q+4}{q^2-q} \qquad \mbox{as } m \to \infty.
    \end{align*}
    Thus for $R := \lim_{m \to \infty} R_m$ and $\delta:=\lim_{m \to \infty} \delta_m$, we obtain
    \[
    R = \frac{2q-4}{q^2-q} \quad \mbox{and} \quad \delta \ge \frac{q^2-4q+4}{q^2-q}.
    \]
    Therefore, 
    \[
    R + \delta \ge \frac{2q-4}{q^2-q} + \frac{q^2-4q+4}{q^2-q} = 1 - \frac{1}{q-1}
    \]
    that is, the sequence $\left( \CL(D^{(m)}, G^{(m)}) \right)_{m \ge 2}$ attain the Tsfasman–Vlăduţ–Zink bound (\ref{eq:TVZ_bound}).
\end{proof}

\begin{theorem}
		\label{LCPinTm}
		Let $q>2$ and $D^{(m)} := (t)_0^{T_m}$ be the zero divisor of $t=(x_1^{q^2}-x_1)/(x_1^q+x_1) \in T_m$. Suppose that there are two divisors $G^{(m)}, H^{(m)} \in \Div(T_m)$ satisfying
		\begin{enumerate}[label=(\roman*)]
			\item \label{cond:LCPinTmC1} $\deg(G^{(m)}), \deg(H^{(m)}) < q^{m+1}-q^m$;
			\item \label{cond:LCPinTmC2} 
 $\gcd(G^{(m)},\,  H^{(m)}) = \sum_{\alpha \in \Omega^*} (q^{m-1}-1)P_\alpha^{(m)} -  \sum_{k=1}^{\floor*{(m-2)/2} } A_{k}^{(m)} + (q^{m-1} - q^{\ceil*{m/2}}-1)P_\infty^{(m)}$; and
			\item \label{cond:LCPinTmC3} 
            \[
            \lmd\! \left(G^{(m)}, \, H^{(m)} \right) = \sum_{\alpha \in \Omega^*} (q^{m}-q^{m-1}-1)P_\alpha^{(m)} -  \sum_{k=1}^{\floor*{(m-2)/2} } A_{k}^{(m)} + (2q^{m} - q^{m-1} - q^{\ceil*{m/2}}-1)P_\infty^{(m)}.
            \]
		\end{enumerate}
		Then, $\left( \CL(D^{(m)}, G^{(m)}), \, \CL(D^{(m)}, H^{(m)}) \right)$ is an LCP of codes over $\F_{q^2}$.
	\end{theorem}
	\begin{proof}
		Let $G^{(m)}, H^{(m)} \in \Div(T_m)$ satisfying \ref{cond:LCPinTmC1}, \ref{cond:LCPinTmC2} and \ref{cond:LCPinTmC3}. From Theorem \ref{thm:thmlcp3.5}, we need to check that $\gcd\!\left(G^{(m)}, \, H^{(m)} \right)$ and $\lmd\!\left(G^{(m)}, \, H^{(m)} \right)-D^{(m)}$ are both non-special divisors of degree $g_m-1$.
		By Proposition \ref{prop:nonspecialTmGeral} and \ref{cond:LCPinTmC2}, 
		$$
		\gcd\!\left(G^{(m)}, \, H^{(m)} \right) = \sum_{\alpha \in \Omega^*} (q^{m-1}-1)P_\alpha^{(m)} -  \sum_{k=1}^{\floor*{(m-2)/2} } A_{k}^{(m)} + (q^{m-1} - q^{\ceil*{m/2}}-1)P_\infty^{(m)}
		$$ 
		is a non-special divisor of degree $g_m-1$. By \ref{cond:LCPinTmC3},
		 we have 
         \allowdisplaybreaks
		\begin{align*}
		 \lmd\!\left(G^{(m)}, \, H^{(m)} \right) - D^{(m)} & =  \lmd\!\left(G^{(m)}, \, H^{(m)} \right) - (t)_0^{T_m} 
        \\
		& \sim  \lmd\!\left(G^{(m)}, \, H^{(m)} \right) - (t)_\infty^{T_m} \\
		& = 
        \sum_{\alpha \in \Omega^*} (q^{m}-q^{m-1}-1)P_\alpha^{(m)} -  \sum_{k=1}^{\floor*{(m-2)/2} } A_{k}^{(m)} \\
		& \qquad  + (2q^{m} - q^{m-1} - q^{\ceil*{m/2}}-1)P_\infty^{(m)} - (q^{m+1}-q^m)P_\infty^{(m)} 
        \\
		& = 
        \sum_{\alpha \in \Omega^*} (q^{m-1}-1)P_\alpha^{(m)} - \! \sum_{k=1}^{\floor*{(m-2)/2} } \!  A_{k}^{(m)}+(q^{m-1}\! -q^{\ceil*{m/2}}-1)P_\infty^{(m)} \\
        &  \qquad  + \sum_{\alpha \in \Omega^*} (q^{m}-2q^{m-1})P_\alpha^{(m)} - (q^{m+1}  - 3q^m + 2q^{m-1})P_\infty^{(m)} 
         \\
		& = 
            \gcd\!\left(G^{(m)}, \, H^{(m)} \right) +  \sum_{\alpha \in \Omega^*} (q-2)q^{m-1}P_\alpha^{(m)} \\
         &\qquad- (q-2)(q-1)q^{m-1}P_\infty^{(m)}\\
		& = \gcd\!\left(G^{(m)}, \, H^{(m)} \right) + \left( (x_1^{q-1}+1)^{q-2} \right)^{T_m}. 
		\end{align*}
		Thus, $\lmd\!\left(G^{(m)}, \, H^{(m)} \right) - D^{(m)}$ is equivalent to $\gcd\!\left(G^{(m)}, \, H^{(m)} \right)$ and, hence, it is also a non-special divisor of degree $g_m-1$. Therefore, we obtain that the codes $\CL(D^{(m)}, G^{(m)})$ and $\CL(D^{(m)}, H^{(m)})$ form an LCP of codes over $\F_{q^2}$.
	\end{proof}
	
	\begin{theorem}
		\label{LCPinTm2}
		Let $q \ge 4$ and $D^{(m)} := (t)_0$. Consider the divisors 
		\begin{align*}
			G^{(m)} & = \sum_{\alpha \in \Omega^*}(q^{m-1}-1) P_\alpha^{(m)} - \sum_{k =1}^{\floor*{(m-2)/2}} A_{k}^{(m)} + (2q^{m} - q^{m-1} - q^{\ceil*{m/2}}-1) P_\infty^{(m)}, \quad \mbox{and} \quad \\
			H^{(m)} & = \sum_{\alpha \in \Omega^*} (q^m-q^{m-1}-1) P_\alpha^{(m)} - \sum_{k =1}^{\floor*{(m-2)/2}} A_{k}^{(m)} + (q^{m-1}-q^{\ceil*{m/2}}-1) P_\infty^{(m)}.
		\end{align*}
		Then, $\left( \CL(D^{(m)}, G^{(m)}), \, \CL(D^{(m)}, H^{(m)}) \right)$ is an LCP of codes over $\F_{q^2}$, where 
		\[ 
        \CL(D^{(m)}, G^{(m)}) \mbox{ is a } [q^{m+1}-q^m, 2q^m - 2q^{m-1}]\mbox{-code} 
        \]
        and \[
        \CL(D^{(m)}, H^{(m)}) \mbox{ is a } [q^{m+1}-q^m, q^{m+1}-3q^m+2q^{m-1}] \mbox{-code}.
        \]
        Moreover, both sequences $\CL(D^{(m)}, G^{(m)})$ and $\CL(D^{(m)}, H^{(m)})$ attain the Tsfasman-Vlăduţ-Zink bound.
    \end{theorem}
		\begin{proof}
		Observe that for all $q \ge 2$ we have the inequalities
		$$
		q^{m-1}-1 \le q^m-q^{m-1}-1 \quad \mbox{and} \quad q^{m-1}-q^{\ceil*{m/2}}-1 \le 2q^m-q^{m-1}-q^{\ceil*{m/2}}-1. 
		$$
		 Thus, 
		$$
		\gcd(G^{(m)},\,  H^{(m)}) = \sum_{\alpha \in \Omega^*} (q^{m-1}-1)P_\alpha^{(m)} -  \sum_{k=1}^{\floor*{(m-2)/2} } A_{k}^{(m)} + (q^{m-1} - q^{\ceil*{m/2}}-1)P_\infty^{(m)}
		$$
		and 
		$$
		\lmd\!\left(G^{(m)}, \, H^{(m)} \right) = \sum_{\alpha \in \Omega^*} (q^{m}-q^{m-1}-1)P_\alpha^{(m)} -  \sum_{k=1}^{\floor*{(m-2)/2} } A_{k}^{(m)} + (2q^{m} - q^{m-1} - q^{\ceil*{m/2}}-1)P_\infty^{(m)}.
		$$
		Moreover,
			\begin{align*}
			\deg(G^{(m)}) & = (q-1)(q^{m-1}-1)-(q^{\floor*{m/2}} - q) + 2q^m-q^{m-1}-q^{\ceil*{m/2}} - 1 \\
			& = 3q^m-2q^{m-1}-q^{\ceil*{m/2}}-q^{\floor*{m/2}}, 
		\end{align*}
		and 
		\begin{align*}
			\deg(H^{(m)}) & = (q-1)(q^m-q^{m-1}-1) - (q^{\floor*{m/2}}-q) + q^{m-1} -  q^{\ceil*{m/2}} - 1 \\
			& = q^{m+1}- 2q^m + 2q^{m-1}  - q^{\ceil*{m/2}} -  q^{\floor*{m/2}}.
		\end{align*}
		For $q \ge 4$, we have 
		$$
	\deg(G^{(m)})=3q^m -2q^{m-1} - q^{\floor*{m/2}} - q^{\ceil*{m/2}} < q^{m+1} - q^{m}
		$$
		and 
		$$
		\deg(H^{(m)}) =q^{m+1}- 2q^m + 2q^{m-1} -  q^{\floor*{m/2}} - q^{\ceil*{m/2}} < q^{m+1}-q^m.
		$$	
	Therefore, we obtain that $\left( \CL(D^{(m)}, G^{(m)}), \, \CL(D^{(m)}, H^{(m)}) \right)$ is an LCP of codes over $\F_{q^2}$ from Theorem \ref{LCPinTm}.
   
    On the other hand, 
    setting $R_m := R(\CL(D^{(m)}, G^{(m)}))$, $\delta_m:= \delta(\CL(D^{(m)}, G^{(m)}))$, $R_m' := R(\CL(D^{(m)}, H^{(m)}))$ and $\delta_m' := \delta(\CL(D^{(m)}, H^{(m)}))$, we obtain that
  \begin{align*}
       \bullet \, R_m & = \frac{ 2q^{m} - 2q^{m-1} }{ q^{m+1} - q^m } \longrightarrow \frac{2}{q} \qquad \mbox{as } m \to \infty,\\
		\bullet \, \delta_m & \ge \frac{q^{m+1}-q^m - (3q^m -2q^{m-1} - q^{\floor*{m/2}} - q^{\ceil*{m/2}})}{q^{m+1} - q^m } \\
          & = \frac{q^{m+1}- 4q^m + 2q^{m-1} + q^{\floor*{m/2}} + q^{\ceil*{m/2}})}{q^{m+1} - q^m } \longrightarrow 1 - \frac{3q-2}{q^2-q} \quad \mbox{ as } m \to \infty,\\  
        \bullet \,  R_m' &= \frac{q^{m+1}-3q^m+2q^{m-1}}{q^{m+1}-q^m} \longrightarrow \frac{ q^2-3q + 2 }{ q^2 - q} \qquad \mbox{as } m \to \infty,\\    
       \bullet  \, \delta_m' & \ge \frac{q^{m+1}-q^m - (q^{m+1}- 2q^m + 2q^{m-1} -  q^{\floor*{m/2}} - q^{\ceil*{m/2}} ) }{q^{m+1} - q^m } \\
         & = \frac{q^m  - 2q^{m-1} +  q^{\floor*{m/2}} + q^{\ceil*{m/2}}  }{q^{m+1} - q^m } \longrightarrow \frac{ q - 2 }{ q^2 - q} \qquad \mbox{as } m \to \infty.
        \end{align*}
        Furthermore, let $R:=\lim_{m \to \infty} R_m$, $\delta :=\lim_{m \to \infty} \delta_m$, $R' :=\lim_{m \to \infty} R_m'$ and $\delta':=\lim_{m \to \infty} \delta_m$, we have
        \[
        R + \delta \ge \frac{2}{q} + 1 - \frac{3q-2}{q^2-q} = 1 - \frac{1}{q-1}
        \]
        and 
          \[
        R' + \delta' \ge \frac{ q^2-3q + 2 }{ q^2 - q}  +\frac{ q - 2 }{ q^2 - q} = 1 - \frac{1}{q-1}
        \]
It follows that for $m \to \infty$, the codes $\CL(D^{(m)}, G^{(m)})$ and $\CL(D^{(m)}, H^{(m)})$ are asymptotically good and they meet the Tsfasman-Vlăduţ-Zink bound.
	\end{proof}

\section{Asymptotically Good Self-Dual and Self-Orthogonal AG Codes}

 \label{sec:SeqSelfOrtho/Dual_Tm}

        As a consequence of the results of the previous section, we were also able to describe sequences of self-orthogonal and self-dual AG codes over $\F_{q^2}$ from the Garcia-Stichtenoth tower $\cT=(T_1, T_2, T_3, \dots )$ defined in \ref{eq:GStower2}. 

        Recall that a code $\C \subseteq \F_q^n$ is called \textit{self-dual} if $\C$ equals its dual code $\C^\perp$. The code $\C$ is called \textit{self-orthogonal} if $\C \subseteq \C^\perp$. It is clear that a self-dual code can only exist for even lengths.

        In \cite{stich1988self}, criteria for self-orthogonality and self-duality of AG codes are given. More specifically:
        
        \begin{lemma}\cite[Corollaries 3.2 and 3.3]{stich1988self}
        \label{lemma:self-dual_stich}Let $F/\F_q$ be a function field and $\CL(D, G)$ be an AG code over $\F_q$.
            Assume there is a differential $\omega \in \Omega_F$ with the following properties:
            \[
            (\omega) \ge 2G -D \quad \mbox{and} \quad \res_{P_i}(\omega) = \res_{P_j}(\omega) \neq 0 \mbox{ for all } 1 \le i, j \le n.
            \]
            Then, $\CL(D, G)$ is a self-orthogonal code. If 
            \[
            (\omega)=2G-D \quad \mbox{and} \quad \res_{P_i}(\omega) = \res_{P_j}(\omega) \mbox{ for all } 1 \le i, j \le n, 
            \]
            then $\CL(D,G)$ is a self-dual $[n, n/2, \ge n/2+1-g]$-code.
        \end{lemma}

        We continue working with the divisor  
        \begin{equation}
            \label{eq:D^{(m)}}
             D^{(m)}=(t)_0^{T_m} \quad \mbox{for } t=\frac{x_1^{q^2}-x_1}{x_1^q+x_1}
        \end{equation}
        and a divisor $G^{(m)}$ of $T_m$ whose support avoids that of $D^{(m)}$. From Lemma \ref{CanonicalDivisorTm}, we have a description of the dual code $\CL(D^{(m)}, G^{(m)})^\perp=\CL(D^{(m)}, D^{(m)} - G^{(m)} + W^{(m)})$, where $W^{(m)}$ is the canonical divisor $W^{(m)}=\left( dt/t \right)^{T_m}$. Thus, we can use the Lemma \ref{lemma:self-dual_stich} to describe self-orthogonal and self-dual codes over $\F_{q^2}$ from $T_m$.

	\begin{theorem}
    \label{thm:Self-Orthogonal_in_Tm}
		Let $D^{(m)} = (t)_0^{T_m}$ be as above. For $\alpha \in \Omega^*$ and $1 \le k \le \floor*{(m-2)/2}$, let $v_\alpha, w_{k, \alpha}, v_0, v_\infty \in \Z$ and consider the divisor
		$$
		G^{(m)} = \sum_{\alpha \in \Omega^*} v_\alpha P_\alpha^{(m)} + \sum_{k =1}^{\floor*{(m-2)/2}} \sum_{\alpha \in \Omega^*} w_{k,\alpha} A_{k, \alpha}^{(m)} + v_0(x_1)_0^{T_m}+v_{\infty} P_\infty^{(m)},
		$$
		where $A_{k, \alpha}^{(m)}=\{ Q \in \PP_{T_m} \mid Q \mbox{ is a zero of } x_{k+1}- \alpha \mbox{ in } T_m \}$.
		If 		
          \begin{enumerate}[label=(\roman*)]
		      \item $0 \le   v_\alpha \le (q^m-2)/2$ for all $\alpha \in \Omega^*$;
                \item $0 \le w_{k,\alpha} \le q^{m-2k-1}-1$ for all $\alpha \in \Omega^*$ and $1 \le k \le \floor*{(m-2)/2}$;
                \item $0 \le v_0 \le (q-2)/2$; and
                \item $0 \le v_\infty \le (q^m-2)/2$,
		  \end{enumerate}
          then $\CL(D^{(m)}, G^{(m)})$ is a self-orthogonal code over $\F_{q^2}$.
	\end{theorem}
	\begin{proof}
		From Equation \ref{eq:DualAGCodeinTm}, the dual code $\CL(D^{(m)}, G^{(m)})^\perp$ is given by $\CL(D^{(m)}, H^{(m)})$ where  
		$$H^{(m)}:=W^{(m)}-G^{(m)}+D^{(m)},  \quad \text{ and }W^{(m)}=(dt/t)^{T_m}.$$ 
		The differential $dt/t \in \Omega_{T_m}$ is such that $\res_{P_i}(dt/t)=1$ for all $P_i \in \Supp(D^{(m)})$ and, from Lemma \ref{CanonicalDivisorTm},
		\begin{align*}
		H^{(m)} & = W^{(m)}-G^{(m)}+D^{(m)} \\
		& =	    \sum_{\alpha \in \Omega^*} (q^m-2-v_\alpha) P_\alpha^{(m)} + \sum_{k =1}^{\floor*{(m-2)/2}} \sum_{\alpha \in \Omega^*}\left(2(q^{m-(2k+1)} - 1) - w_{k,\alpha} \right) A_{k, \alpha}^{(m)}  \\
		& \quad + (q-2-v_0) \cdot (x_1)_0^{T_m} + (q^m-2 - v_{\infty}) P_\infty^{(m)}.
    \end{align*}
		Thus, for  
        \[ 
        0 \le v_\alpha, v_\infty \le (q^m-2)/2 
        , \quad 0 \le w_{k,\alpha} \le q^{m-2k-1}-1 \quad  \mbox{and} \quad 0 \le v_0 \le (q-2)/2
        \]
        for any $\alpha \in \Omega^*$ and $1 \le k \le \floor*{(m-2)/2}$, we obtain that $G^{(m)} \le H^{(m)}$. From Lemma \ref{lemma:self-dual_stich}, we conclude
		\[
        \CL(D^{(m)}, G^{(m)}) \subseteq \CL(D^{(m)}, H^{(m)})=\CL(D^{(m)}, G^{(m)})^\perp.
            \]
	\end{proof}

    As a consequence of the theorem, we have that
    
    \begin{theorem}
        \label{thm:self-orthogonal_in_Tm}
        Let $q \ge 4$. Let $D^{(m)}= (t)_0^{T_m}$ be as in Theorem \ref{thm:Self-Orthogonal_in_Tm}. Consider the divisor $G^{(m)}$ of $T_m$ given by
        \[
        G^{(m)} = \sum_{\alpha \in \Omega^*} 2(q^{m-1}-1) P_\alpha^{(m)} + 2(q^{m-1}-q^{\ceil*{m/2}}-1)P_\infty^{(m)}.
        \]
        Then, $\C_m := \CL(D^{(m)}, G^{(m)})$ is a self-orthogonal code over $\F_{q^2}$ and, for $m \ge 4$, $\C_m$ has parameters
        \[
        [q^{m+1}-q^m, \, q^m - q^{\ceil*{m/2}} + q^{\floor*{m/2}} -2q, \, \ge q^{m+1}-3q^m + 2q^{\ceil*{m/2}} + 2q].
        \]
        In particular, the rate $R(\C_m)$ and the relative minimum distance $\delta(\C_m)$ of the code $\CL(D^{(m)}, G^{(m)})$ satisfy  
        \[
        R(\C_m) = \frac{q^m - q^{\ceil*{m/2}} + q^{\floor*{m/2}} -2q}{q^{m+1} - q^m}
        \quad \mbox{and} \quad \delta(\C_m) \ge \frac{q^{m+1}-3q^m + 2q^{\ceil*{m/2}} + 2q}{q^{m+1} - q^m}.
        \]
        Furthermore,
        \[
        \lim_{m \to \infty} R(\C_m) = \frac{1}{q-1} \quad \mbox{and} \quad \lim_{m \to \infty} \delta(\C_m) \ge  1-\frac{2}{q-1},
        \]
        that is,
        the sequence $\left( \C_m \right)_{m \ge 2}$ of codes is asymptotically good over $\F_{q^2}$ and it meets the Tsfasman-Vlăduţ-Zink bound for $q \ge 7$.
    \end{theorem}
\begin{proof}
Initially, observe that for $q \ge 4$ we have that for all  $m \ge 2$
\begin{equation}
\label{eq:self-orthogonalTm}
2(q^{m-1}-1) < \frac{q^m-2}{2}
\qquad \mbox{and}  \qquad 2(q^{m-1}-q^{\ceil*{m/2}}-1) < \frac{q^m-2}{2}. 
\end{equation}
 We are going to apply Theorem \ref{thm:Self-Orthogonal_in_Tm}. For this, let
\begin{enumerate}[label=]
    \item $v_\alpha = 2(q^{m-1}-1)$ for all $\alpha \in \Omega^*$;
    \item $v_\infty=2(q^{m-1}-q^{\ceil*{m/2}}-1)$ and ;
    \item $v_0 = 0$ and $w_{k,\alpha}=0$ for all $\alpha \in \Omega^*$ and $1 \le k \le \frac{m-2}{2}$.
\end{enumerate} 
Now we consider the divisor $G^{(m)}$ as given in in Theorem \ref{thm:Self-Orthogonal_in_Tm}, that is, 
\begin{align*}
    G^{(m)} & = \sum_{\alpha \in \Omega^*} 2(q^{m-1} - 1) P_\alpha^{(m)} + 2(q^{m-1}-q^{\ceil*{m/2}}-1)P_\infty^{(m)} \\
 & = \sum_{\alpha \in \Omega^*} v_\alpha P_\alpha^{(m)} + v_\infty P_\infty^{(m)}\;.
\end{align*}

From the above and Equation (\ref{eq:self-orthogonalTm}), we have the values $v_\alpha$, $w_{k, \alpha}$, $v_0$ and $v_\infty$ satisfy the Conditions $(i)-(iv)$ in Theorem \ref{thm:Self-Orthogonal_in_Tm} for all $\alpha \in \Omega^*$ and $1 \le k \le \floor*{(m-2)/2}$.
Hence, the code $\C_m=\CL(D^{(m)}, G^{(m)})$ is a self-orthogonal code over $\F_{q^2}$. 

On the other hand, for $q \ge 4$ and $m \ge 4$, the degree of the divisor $G^{(m)}$ is given by
\begin{align*}
    \deg( G^{(m)} ) & = 2(q-1)(q^{m-1}-1) + 2(q^{m-1}-q^{\ceil*{m/2}} - 1) \\
    & = 2q^m - 2q^{\ceil*{m/2}} - 2q^{\floor*{m/2}} + (2q^{\floor*{m/2}}- 2q) \\
    & = 2g_m - 2 + (2q^{\floor*{m/2}}- 2q) \\
    & > 2g_m - 2.
\end{align*}

Thus, by Proposition \ref{prop:parametersAGcodes}, the dimension $k_m$ and the minimum distance $d_m$ of $\C_m$ satisfy 
\begin{align*}
    k_m & = \deg(G^{(m)}) + 1 - g_m \\
    & = q^m + q^{\floor*{m/2}} - q^{\ceil*{m/2}}-2q     
\end{align*}
and
\[ 
d_m \ge q^{m+1}-q^m - (2q^m - 2q^{\ceil*{m/2}} - 2q) = q^{m+1} - 3q^m + 2q^{\ceil*{m/2}} + 2q.
\]
In particular,
\[
R := \lim_{m \to \infty} R(\C_m) = \lim_{m \to \infty} \frac{k_m}{n_m}  = \lim_{m \to \infty} \frac{q^m + q^{\floor*{m/2}} - q^{\ceil*{m/2}}-2q }{q^{m+1}-q^m} = \frac{1}{q-1}
\]
and
\[
\delta := \lim_{m \to \infty} \delta(\C_m) = \lim_{m \to \infty} \frac{d_m}{n_m} \ge \lim_{m \to \infty} \frac{q^{m+1} - 3q^m + 2q^{\ceil*{m/2}} + 2q}{q^{m+1}-q^m} \ge 1 - \frac{2}{q-1},
\]
that is,
\[
R + \delta \ge 1 - \frac{1}{q-1}.
\]
We conclude the sequence $(\C_m)_{m \ge 2}$ meets the Tsfasman-Vlăduţ-Zink bound for $q \ge 7$.
\end{proof}

Finally, for $q$ even we can describe self-dual codes over $\F_{q^2}$. 
    
	\begin{theorem}
	    \label{thm:self-dual_in_Tm}
		Let $q \geq 4$ even and  $D^{(m)}$ be as in Theorem \ref{thm:Self-Orthogonal_in_Tm}. Then, for 
		$$
		G^{(m)} = \sum_{\alpha \in \Omega^*} \frac{q^m-2}{2} P_\alpha^{(m)} + \sum_{k =1}^{\floor*{(m-2)/2}} \sum_{\alpha \in \Omega^*} (q^{m-2k-1}-1) A_{k, \alpha}^{(m)} + \frac{q-2}{2} \left( x_1 \right)_0^{T_m}+\frac{q^m-2}{2} P_\infty^{(m)},
		$$
		the code $\CL(D^{(m)}, G^{(m)})$ is a self-dual $[q^{m+1}-q^m, \, (q^{m+1}-q^m)/2]$-code over $\F_{q^2}$.
	\end{theorem}
    \begin{proof}
		We have that $G^{(m)}$ is equal to $H^{(m)}$, for $H^{(m)}$ as in the proof of Theorem \ref{thm:Self-Orthogonal_in_Tm}. Therefore, 
		$$
		\CL(D^{(m)},G^{(m)})=\CL(D^{(m)},H^{(m)}) = \CL(D^{(m)},G^{(m)})^\perp. 
		$$
	\end{proof}

    \begin{corollary}
        \label{cor:self-dual_in_Tm}
        Let $q \ge 4$ even, $D^{(m)}$ and $G^{(m)}$ be as in Theorem \ref{thm:self-dual_in_Tm}. For $\C_m :=\CL(D^{(m)}, G^{(m)})$, we have that $\left( \C_m \right)_{m \ge 2}$ is a sequence of self-dual codes over $\F_{q^2}$ having parameters $$[q^{m+1}-q^m, (q^{m+1}-q^m)/2, d_m]$$ satisfying
        \[
        \lim_{m \to \infty} \frac{d_m}{q^{m+1}-q^m} \ge \frac{1}{2}- \frac{1}{q-1}.
        \]
       Therefore, for $q \ge 8$, the self-dual codes $\C_m$ are better than the Gilbert-Varshamov bound, in fact they attain the TVZ-bound for $R=1/2$ and 
        $\delta \ge 1-\frac{1}{q-1}-R$. 
    \end{corollary}

\bibliographystyle{siam}

\end{document}